\documentclass[12pt,reqno]{article}

\usepackage{amsmath,amsthm,amssymb,amsfonts}

\usepackage[T1]{fontenc}
\usepackage[utf8]{inputenc}

\usepackage{graphicx}
\usepackage[usenames]{color}
\usepackage{verbatim}

\usepackage{authblk}

\usepackage{hyperref}

\theoremstyle{plain}
\newtheorem{theorem}{Theorem}

\newtheorem{proposition}[theorem]{Proposition}

\theoremstyle{definition}

\theoremstyle{remark}

\newtheorem*{remark*}{Remark}

\begin{document}

\title{Lonely Solids}

\author{Ivo Fagundes David de Oliveira, Tanya Khovanova, and Yogev Shpilman}

\maketitle

\begin{abstract}
A three-dimensional solid has the Rupert property if a congruent copy of the solid can pass through a hole cut through it without splitting it. We extend this idea to pairs of convex solids: two solids are called \textit{friends} if each can pass through a suitable hole in the other. A solid is called \textit{lonely} if it has no friends, including itself.

We show that a convex solid is lonely if and only if it has constant width. We also show that every convex solid that does not have constant width has a particularly simple friend: an arbitrarily long and arbitrarily thin rectangular cuboid. Finally, we prove that all non-constant-width convex solids lie in a single connected component of the friendship graph. More precisely, any two such solids are connected by a chain of at most ``three handshakes''.
\end{abstract}

\section{Introduction}

This story started many years ago with Prince Rupert's cube, the largest cube that can pass through a hole in a unit cube without splitting it into separate pieces. Surprisingly, Rupert's cube has a side length of $\frac{3\sqrt{2}}{4}\approx 1.06$. This means that a larger cube can pass through a smaller cube, as shown in Figure~\ref{fig:Rupertcube}.

\begin{figure}[ht!]
    \centering
    \includegraphics[width=250pt]{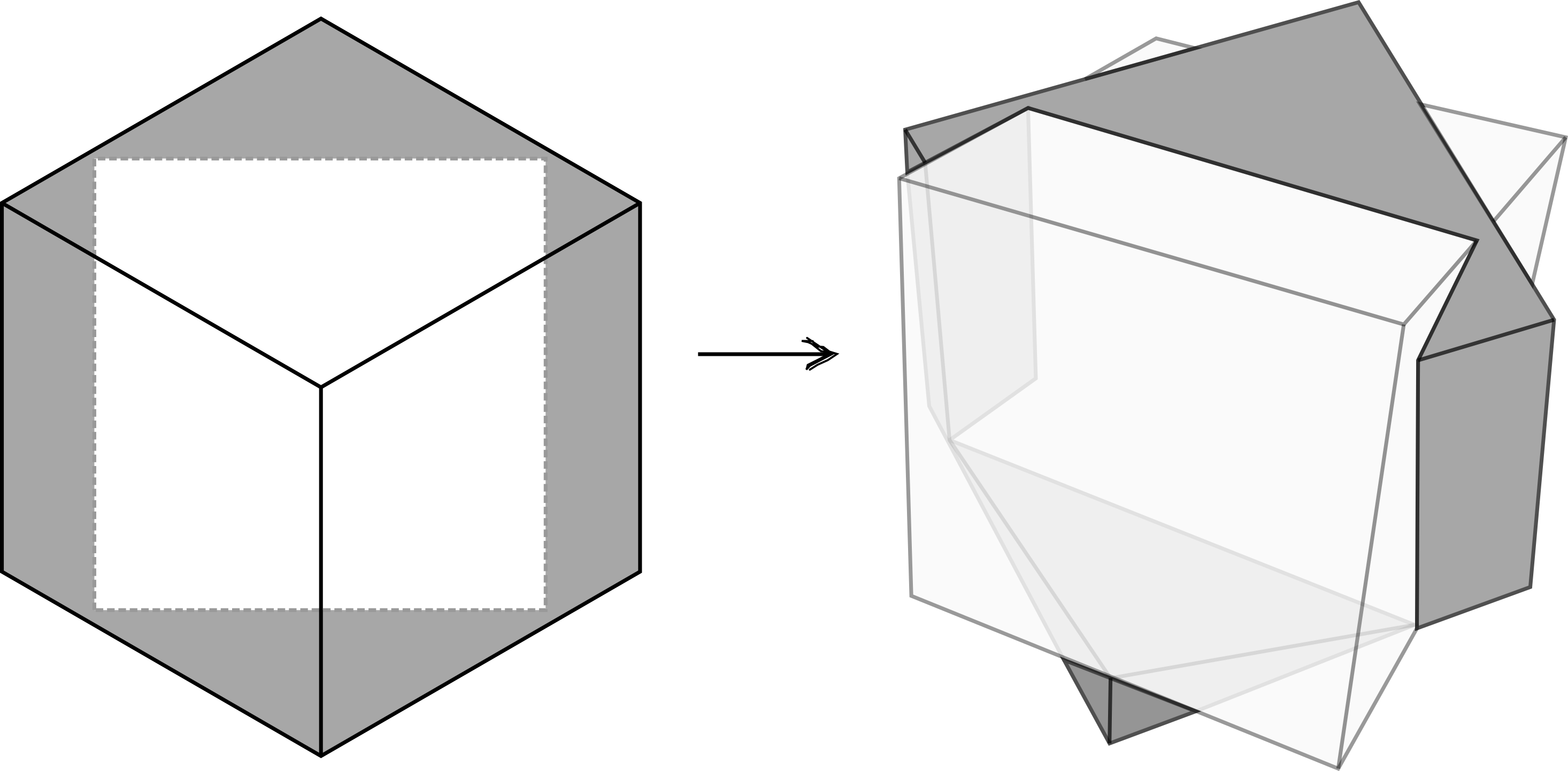}
    \caption{Rupert cube}
    \label{fig:Rupertcube}
\end{figure}

This motivates the following definition. A solid shape has the \textit{Rupert property} if its copy of the same size can pass through itself by a rigid translation along some direction, without separating it into disconnected pieces. It can be shown that any Platonic solid has the Rupert property. In August 2025, a preprint by Jakob Steininger and Sergey Yurkevich presented a convex polyhedron that does not have the Rupert property \cite{SY2025}. This polyhedron is now called a \textit{noperthedron}.

We now move to two solids. We say that solid $A$ can \textit{Rupert through} $B$ if $A$ can pass through $B$ without splitting $B$.

This story is continued in our book \cite{PuzzleBook2026}, where we proposed the following puzzle:

\begin{quote}
    Is it possible to drill a square hole in a sphere and a circular hole in a cube in such a way that the cube can pass through the sphere, and the sphere can pass through the cube?
\end{quote}

The answer to this puzzle is no, as proved in the book. However, we also introduced the following definition there: We say that two convex solids $A$ and $B$ are \textit{friends} if they can Rupert through each other. As we mentioned before, all Platonic solids are friends with themselves, and the noperthedron is not a friend of itself. Furthermore, we say that a convex solid $A$ is \textit{lonely} if it has no friends, including itself: really lonely.

We also showed in the book that a sphere turns out to be a lonely solid, and we ended this topic with an open question: is it possible to come up with another lonely solid? This paper is an answer to this open question.

Before answering, though, we need to add some definitions. The \textit{width} of a solid in a particular direction is the distance between two opposite parallel supporting planes that are perpendicular to the given direction. A \textit{solid of constant width} is a convex solid having the same width in every direction.

A sphere has constant width. But there are many more solids of constant width. One way to generate such a solid is to rotate a Reuleaux triangle about one of its axes of symmetry. It is known that among all bodies of revolution with the same constant width, the one with minimum volume is provided by such a shape \cite{CampiColesantiGronchi1996}.

There are also Meissner tetrahedra: two noncongruent shapes that can be produced by slightly modifying the Reuleaux tetrahedron \cite{MeissnerSchilling1912}. The Meissner tetrahedra are conjectured to have the minimum volume among all bodies of a given constant width \cite{BonnesenFenchel1934}.

\section{A Special Friend}

Let $L,h,\varepsilon>0$ satisfy $L \gg h \gg \varepsilon$, and let $R=R(L,h,\varepsilon)$ be a rectangular cuboid of dimensions $L \times h \times \varepsilon$, which we refer to as a \textit{social butterfly}. We refer to $L$ as the \textit{length}, and $\varepsilon$ as the \textit{thickness} of the cuboid.

\begin{proposition}
\label{prop:socbutt}
    Any non-constant-width convex solid has a social butterfly $B$ as a friend with arbitrarily large length and arbitrarily small thickness.
\end{proposition}

\begin{proof}
Assume solid $A$ does not have constant width. Let $d^+$ and $d^-$ be its maximum and minimum widths. Consider direction $D^-$, so that the two supporting planes perpendicular to $D^-$ are at distance $d^-$. Consider a direction $D_1$, perpendicular to $D^-$ such that the two supporting planes perpendicular to $D_1$ are at a distance $d_1$, where $d^- \le d_1 \le d^+$.

Choose $D$ perpendicular to both $D^-$ and $D_1$; then the projection along $D$ lies inside the rectangle determined by the two pairs of supporting planes with side lengths $d^-$ and $d_1$. It follows that the projection of $A$ along $D$ is contained in a rectangle with sides $d^-$ and $d_1$.

Define a sufficiently thin rectangular cuboid $B = R(L,h,\varepsilon)$ of height $h$ such that $d^+ > h > d^-$ and length $L > d^+$, so the dimensions of this thin cuboid are $L$, $h$, and $\varepsilon$, where $\varepsilon$ is very small.

The rectangle with sides $d^-$ and $d_1$ fits into the $L\times h$ face of $B$: with $d^-$ along the side of length $h$, and $d_1$ along the side of length $L$. Since $h > d^-$ and $L > d^+ \ge d_1$, if we make such a hole in $B$ across the shortest direction, we will not split $B$, so solid $A$ can Rupert through $B$.

All that is needed is to show that $B$ can Rupert through $A$. To see this, notice that $B$ has a projection that is a rectangle with sides $h$ and $\varepsilon$. We can fit a segment of length $d^+$ along the longest direction in $A$. Since $h < d^+$, a slightly shorter segment of length $h$ can be placed in the interior of the projection, so for sufficiently small $\varepsilon$ a rectangle $h\times\varepsilon$ fits. Thus, $B$ can Rupert through $A$.
\end{proof}

\section{Lonely Solids}

\begin{theorem}
\label{thm:lonely}
A convex solid is lonely if and only if it is of constant width. 
\end{theorem}

\begin{proof} 
From Proposition~\ref{prop:socbutt}, we know that if a convex solid does not have constant width, it has a friend. What is left to prove is that constant-width solids are lonely.

Suppose our solid $A$ has a constant width $d$. Assume by contradiction that there exists some solid $B$ that is a friend of $A$. Suppose $B$ Ruperts through $A$ in some direction. It follows that the projection of $B$ in this direction can fit inside the projection of $A$ in this direction.  Since the hole cannot reach all the way to the boundary without separating $A$, the cross-section of the tunnel has width $d'<d$ in every direction of this plane $P_1$. Thus, while passing through $A$, the solid $B$ is contained in an infinite ``cylinder'' $C$ whose cross-section has width $d'<d$ in every direction in $P_1$.

Since $A$ can Rupert through $B$, it must also be able to go through this cylinder $C$. Since the tunnel must leave enough material to keep $A$ connected, the cross-section available to $A$ has to have width at least $d$ in every direction of the corresponding plane $P_2$.

The two planes $P_1$ and $P_2$ have an intersecting direction, in which the cylinder $C$ has to have a width less than $d$ and at least $d$ at the same time, creating a contradiction.
\end{proof}

\section{Three Handshakes}

We say that two friends are \textit{one handshake} away from each other.

In this final section, we complete the classification of the ``social behavior'' of convex solids by showing that all convex solids that are not of constant width belong to a single connected component under mutual Rupertability. Moreover, any two such solids are at most three handshakes away from each other.

\begin{theorem}
\label{thm:3hs}
Let $A$ and $B$ be two convex solids that are not of constant width. Then we can find a friend of $A$, denoted $F_A$, and a friend of $B$, denoted $F_B$, such that they are friends of each other.
\end{theorem}

\begin{proof}
Proposition~\ref{prop:socbutt} implies that we can find a friend $F_A$ of $A$ with dimensions $L_A > h_A > \varepsilon_A$ and a friend $F_B$ of $B$ with dimensions $L_B > h_B > \varepsilon_B$.

By the same proposition, we can expand the length and shorten the thickness of a social butterfly. Let $L = \max\{L_A,L_B\}$ and $\varepsilon = \min\{\varepsilon_A,\varepsilon_B\}$. Thus, we can assume that $F_A = R(L,h_A,\varepsilon)$ and $F_B = R(L,h_B,\varepsilon)$.

Observe that $F_A$ and $F_B$ are friends with each other. Indeed, to pass $F_A$ through $F_B$, use the projection of $F_A$ with sides $h_A$ and $\varepsilon$, place the side $h_A$ along the long side $L$ of $F_B$, and place the tiny side $\varepsilon$ along the side $h_B$. The same argument works in the other direction.
\end{proof}

Consequently, two convex solids that are not of constant width are at distance at most three in the friendship graph.

\begin{remark*}
Theorems~\ref{thm:lonely} and~\ref{thm:3hs} show that solids of constant width are not only lonely, but also isolated from a unique large ``social'' component consisting of all other convex solids.
\end{remark*}


\begin{thebibliography}{9}

\bibitem{BonnesenFenchel1934}
T.~Bonnesen and W.~Fenchel,
{\em Theorie der konvexen Körper},
Julius Springer, Berlin, 1934.

\bibitem{CampiColesantiGronchi1996}
S.~Campi, A.~Colesanti, and P.~Gronchi, Minimum problems for volumes of convex bodies, in \textit{Partial differential equations and applications}, Vol. 177 of Lecture Notes in Pure and Appl.\ Math., Dekker, New York, 1996, pp.\ 43--55.

\bibitem{MeissnerSchilling1912}
E.~Meissner and F.~Schilling,
Drei Gipsmodelle von Flächen konstanter Breite,
{\em Zeitschrift für angewandte Mathematik und Physik} {\bf 60}, 1912, 92--94.

\bibitem{PuzzleBook2026} I.~F.~D.~de Oliveira, T.~Khovanova, and Y.~Shpilman, {\em Mathematical Puzzles and Curiosities}, World Scientific, 2026.

\bibitem{SY2025} J.~Steininger and S.~Yurkevich, A convex polyhedron without Rupert's property, arXiv:2508.18475 [math.MG], 2025.

\end{thebibliography}
\end{document}